\documentclass[12pt]{amsart}
\usepackage{amssymb,amsmath,amsfonts,latexsym}
\usepackage{bm}

\newcommand{\newsection}[1]{\setcounter{equation}{0} \section{#1}}
\newcommand{\bea}{\begin{eqnarray}}
\newcommand{\eea}{\end{eqnarray}}

\newcommand{\vp}{\varphi}

\newcommand{\clh}{\mathcal{H}}

\newcommand{\clq}{\mathcal{Q}}

\newcommand{\cls}{\mathcal{S}}

\newcommand{\clz}{\mathcal{Z}}

\newcommand{\D}{\mathbb{D}}
\newcommand{\bS}{\mathbb{S}}

\newcommand{\raro}{\rightarrow}

\def\textmatrix#1&#2\\#3&#4\\{\bigl({#1 \atop #3}\ {#2 \atop #4}\bigr)}
\def\dispmatrix#1&#2\\#3&#4\\{\left({#1 \atop #3}\ {#2 \atop #4}\right)}
\newcommand{\be}{\begin{equation}}
\newcommand{\ee}{\end{equation}}
\newcommand{\ben}{\begin{eqnarray*}}
\newcommand{\een}{\end{eqnarray*}}

\newcommand{\NI}{\noindent}

\newcommand{\bi}{\begin{itemize}}
\newcommand{\ei}{\end{itemize}}

\newcommand{\T}{\mathbb{T}}

\newtheorem{Question}{\sc Problem}

\theoremstyle{definition}

\theoremstyle{plain}

\newtheorem{thm}{Theorem}[section]
\newtheorem{cor}[thm]{Corollary}
\newtheorem{lem}[thm]{Lemma}
\newtheorem{prop}[thm]{Proposition}
\theoremstyle{definition}
\newtheorem{defn}[thm]{Definition}

\newtheorem{ex}[thm]{Example}

\numberwithin{equation}{section}

\let\phi=\varphi

\begin{document}

\title[Pairs of inner projections]{Pairs of inner projections and two applications}

\author[Debnath]{Ramlal Debnath}
\address{Department of Mathematics, Indian Institute of Technology Bombay, Powai, Mumbai,
400076, India}
\email{ramlaldebnath100@gmail.com, ramlal@math.iitb.ac.in}

\author[Pradhan] {Deepak K. Pradhan}
\address{Department of Mathematics, Indian Institute of Technology Hyderabad, Kandi, Sangareddy, Telangana, 502284,
India}
\email{deepak12pradhan@gmail.com, dkpradhan@math.iith.ac.in}

\author[Sarkar]{Jaydeb Sarkar}
\address{Indian Statistical Institute, Statistics and Mathematics Unit, 8th Mile, Mysore Road, Bangalore, 560059,
India}
\email{jay@isibang.ac.in, jaydeb@gmail.com}




\begin{abstract}
Orthogonal projections onto closed subspaces of $H^2(\mathbb{D}^n)$ of the form $\varphi H^2(\mathbb{D}^n)$ for inner functions $\varphi$ on $\mathbb{D}^n$ are referred to as inner projections, where $H^2(\mathbb{D}^n)$ denotes the Hardy space over the open unit polydisc $\mathbb{D}^n$. In this paper, we classify pairs of commuting inner projections. We also present two seemingly independent applications: the first is an answer to a question posed by R. G. Douglas, and the second is a complete classification of partially isometric truncated Toeplitz operators with inner symbols on $\mathbb{D}^n$.
\end{abstract}

\maketitle

\tableofcontents

\newsection{Introduction}\label{sec: 1}

This paper concerns commuting pairs of orthogonal projections blended with analytic flavors and two sides of applications of seemingly disparate problems. The first application yields an answer (in a much broader context) to a question raised by R. G. Douglas (see Bickel and Liaw \cite[page 104]{BL}), whereas the second classifies partially isometric truncated Toeplitz operators with inner symbols on the polydisc. Recall that a bounded linear operator $P$ acting on a Hilbert space $\clh$ is called an \textit{orthogonal projection} (\textit{projection} in short) if
\[
P = P^* = P^2.
\]
On one hand, the structure of projections is simple and quite standard, and on the other hand, the linear analysis relies heavily on the projections of Hilbert spaces. The structure of pairs of commuting projections is also simple and well-known. For instance, for a pair of projections $P_1$ and $P_2$ acting on a Hilbert space, the operator $P_1 P_2$ is a projection if and only if
\[
P_1 P_2 = P_2 P_1.
\]
Moreover, in this case, the orthogonal projection $P_1 P_2$ is given by
\[
P_1 P_2 = P_{\text{ran} P_1 \cap \text{ran} P_2}.
\]
Throughout the paper, given a Hilbert space $\clh$ and a closed subspace $\cls \subseteq \clh$, we denote by $P_{\cls}$  the associated (orthogonal) projection onto $\cls$.

We are also interested in pairs of commuting projections. Our point of view, however, is more analytic in nature and relates to some of the classical concepts in Hilbert function spaces. Let us now explain the analytic component of our consideration. Given a natural number $n$, we denote by $\D^n$ the open unit polydisc in $\mathbb{C}^n$. We let $H^\infty(\D^n)$ denote the Banach algebra of all bounded analytic functions defined on $\D^n$, that is
\[
H^\infty(\D^n) = \{ \vp \in \text{Hol}(\D^n): \|\vp\|_\infty = \sup_{z \in \D^n} |\vp(z)| < \infty\}.
\]
A function $\vp \in H^\infty(\D^n)$ is said to be \textit{inner} if
\[
|\vp(z)| = 1,
\]
for all $z \in \T^n$ a.e, where $\T^n = \partial \D^n$ is the distinguished boundary of $\D^n$. The boundary value of $\vp$ is in the sense of Fatou \cite{Rud}: a bounded analytic function defined on $\D^n$ has radial limits on $\T^n$ a.e. We also recall that the Hardy space $H^2(\D^n)$ is the Hilbert space of all analytic functions $f$ on $\D^n$ such that
\[
\|f\| := \Big(\sup_{0 \leq r < 1} \int_{\mathbb{T}^n} |f(rz_1, \ldots, r z_n)|^2 d{m}({z}) \Big)^{\frac{1}{2}} < \infty,
\]
where $z = (z_1, \ldots, z_n)$ and $dm(z)$ is the normalized Lebesgue measure on $\mathbb{T}^n$. It is well-known that a function $\vp \in H^\infty(\D^n)$ is inner if and only if the analytic Toeplitz operator $M_\vp$ defines an isometry on $H^2(\D^n)$. Recall, in general, for each $\vp \in H^\infty(\D^n)$, $M_\vp$ is a bounded linear operator on $H^2(\D^n)$ with $\|M_\vp\| = \|\vp\|_\infty$, where
\[
M_\vp f = \vp f,
\]
for all $f \in H^2(\D^n)$. Therefore, $\vp H^2(\D^n)$ is a closed subspace of $H^2(\D^n)$ for each inner function $\vp \in H^\infty(\D^n)$. When $\phi(z) = z_i$, $i=1, \ldots, n$, we get the simplest example of an inner function. In particular, $(M_{z_1}, \ldots, M_{z_n})$ is a commuting tuple of isometries on $H^2(\D^n)$. Therefore, given an inner function $\vp \in H^\infty(\D^n)$, the closed subspace $\vp H^2(\D^n)$ is a joint invariant subspace of $(M_{z_1}, \ldots, M_{z_n})$. A closed subspace $\cls$ of $H^2(\D^n)$ is said to be joint invariant if $z_i \cls \subseteq \cls$ for all $i=1, \ldots, n$.  

\begin{defn}
A projection $P$ on $H^2(\D^n)$ is said to be an inner projection if there exists an inner function $\vp \in H^\infty(\D^n)$ such that
\[
\text{ran} P = \vp H^2(\D^n),
\]
or equivalently
\[
P = P_{\vp H^2(\D^n)}.
\]
\end{defn}

The problem that follows is the one that the present paper is mainly concerned with.

\begin{Question}\label{prob 1}
Classify pairs of commuting inner projections.
\end{Question}

In other words, we want to classify inner functions $\vp_1, \vp_2 \in H^\infty(\D^n)$ such that
\[
[P_{\vp_1 H^2(\D^n)}, P_{\vp_2 H^2(\D^n)}] = 0.
\]
Recall that the commutator of a pair of bounded linear operators $A$ and $B$ acting on a Hilbert space is denoted by $[A,B]$, that is, $[A,B] = AB - BA$. Of course, we are looking for a solution to the above problem that is connected with the matching inner functions; that is, we are looking for an analytic answer. To simplify the terminology, we introduce a function-theoretic notion.

\begin{defn}
Let $\Omega \subseteq \mathbb{C}^n$, $n > 1$. Two functions $f, g: \Omega \raro \mathbb{C}$ are said to be separated if $f$ and $g$ do not depend on common variables.
\end{defn}

Given two inner functions $\vp_1, \vp_2 \in H^\infty(\D^n)$, we say that $\vp_1$ divides $\vp_2$ if $\vp_2 = \vp \vp_1$ for some $\vp \in H^\infty(\D^n)$ (in this case, the function $\vp$ will be inner). We are now ready to formulate the solution to the preceding problem:

\begin{thm}\label{thm intro inner}
Let $\vp_1,\vp_2\in H^{\infty}(\D^n)$ be non-constant inner functions. Then the following hold:
\begin{enumerate}
\item $[P_{\vp_1H^2(\D)}, P_{\vp_2H^2(\D)}]=0$ if and only if either $\vp_1$ divides $\vp_2$ or $\vp_2$ divides $\vp_1$.
\item Suppose $n > 1$. Then $[P_{\vp_1H^2(\D^n)}, P_{\vp_2H^2(\D^n)}]=0$ if and only if there exist inner functions $\tilde{\vp}_1, \tilde{\vp}_2, \psi \in H^{\infty}(\D^n)$ such that $\tilde{\vp}_1$ and $\tilde{\vp}_2$ are separated, and such that
\[
\vp_1 = \tilde{\vp}_1 \psi \text{ and } \vp_2 = \tilde{\vp}_2 \psi.
\]
\end{enumerate}
\end{thm}

Note that if one of the inner functions, say $\vp_1$, is a constant function, then $P_{\vp_1H^2(\D^n)} = I_{H^2(\D^n)}$, which immediately implies that $[P_{\vp_1H^2(\D^n)}, P_{\vp_2H^2(\D^n)}]=0$. This observation holds for any $n \geq 1$.

One of the keys to proving the above theorem is the use of the concrete structure of partially isometric Toeplitz operators. More specifically, the proof of the first case uses the classical Brown and Douglas classification of partially isometric Toeplitz operators on $H^2(\D)$ \cite{BH}, whereas the proof of the second case employs the same in several variables, as recently obtained in \cite{DDJ}.

As part of the application, we will apply Theorem \ref{thm intro inner} to address two different kinds of problems that are also purportedly independent of the central question posed in Problem \ref{prob 1}. The problems do have one thing in common: they are related to model spaces (also popularly known as Beurling quotient modules or spaces). Let $\vp \in H^{\infty}(\D^n)$ be a nonconstant inner function. Then the \textit{model space} corresponding to $\vp$ is the closed subspace
\begin{equation}\label{eqn: QM}
\clq_\vp := H^2(\D^n)/ \vp H^2(\D^n) \cong H^2(\D^n) \ominus \vp H^2(\D^n).
\end{equation}
Since $\vp H^2(\D^n)$ is a joint invariant subspace of $(M_{z_1}, \ldots, M_{z_n})$, it readily follows that $\clq_\vp$ is a joint invariant subspace of $(M_{z_1}^*, \ldots, M_{z_n}^*)$.

Our first application answers a question posed by Douglas that has been marked out in the paper by Bickel and Liaw \cite[page 104]{BL} - ``It was suggested to us in private
communications with R. G. Douglas that one should then ask the very general, rather attractive question'':

\begin{Question}[Douglas]\label{Q: 3}
For which inner functions $\vp_1$ and $\vp_2$ in $H^{\infty}(\D^2)$ is the projection $P_{\clq_{\vp_1}} P_{\clq_{\vp_2}}$ finite rank?
\end{Question}

Using Theorem \ref{thm intro inner} (along with some other tools), we completely settle this problem. In fact, we answer this in the most general terms:

\begin{thm}\label{thm intro Douglas}
Let $\vp_1$ and $\vp_2$ be inner functions in $H^{\infty}(\D^n)$, $n \geq 2$. Then the following hold:
\begin{enumerate}
\item If $n=2$, then $P_{\clq_{\vp_1}} P_{\clq_{\vp_2}}$ is a finite rank projection if and only if either one of the following conditions holds:
\begin{enumerate}
\item $\vp_1$ or $\vp_2$ is a constant function;
\item $\vp_1$ and $\vp_2$ are separated and finite Blaschke products.
\end{enumerate}
\item If $n > 2$, then it is impossible for $P_{\clq_{\vp_1}} P_{\clq_{\vp_2}}$ to be a finite rank projection unless $P_{\clq_{\vp_1}} P_{\clq_{\vp_2}} = 0$.
\end{enumerate}
\end{thm}

It is intriguing to notice from a general perspective (as also we will see in the proof of the above result) that the real key to the solution to Douglas's question was concealed in the classic Brown and Douglas paper \cite{BH}. One must, of course, go through a number of nontrivial generalizations and refinements of classical results.

It is important to note that from the case of $n=1$ to the case of $n=2$, and even to the case of $n>2$, the theory of Hilbert function spaces and the theory of commuting tuples of operators differ significantly. Numerous concrete problems are the subject of ongoing research, such as the dilation theory, von Neumann inequality, interpolation problem, model spaces, invariant subspaces. The result above shows again that there is a striking difference between the model spaces when $n=2$ and when $n>2$. In terms of model spaces, the abrupt shift from the $n = 2$ case to the $n > 2$ case was previously noticed in \cite{DGS} as well as in \cite{JS}. When working with Hilbert function spaces on the polydisc, the fundamental reason for such a shift in results is unclear. This could be related to other puzzling issues, such as the failure of von Neumann inquality in more than two variables and the significantly intricate structure of tuples of commuting isometries.

The second application is a classification of truncated Toeplitz operators with inner symbols that are partially isometric. Denote by $L^{\infty}(\mathbb{T}^n)$ the von Neumann algebra of essentially bounded Lebesgue measurable functions on $\mathbb{T}^n$. Let ${\varphi_1} \in H^\infty(\D^n)$ be an inner function and let ${\varphi_2}\in L^{\infty}(\mathbb{T}^n)$.  The \textit{truncated Toeplitz operator} $T^{{\varphi_1}}_{{\varphi_2}}$ on $\clq_{\varphi_1}$ with symbol ${\varphi_2} \in L^{\infty}(\mathbb{T}^n)$ is the bounded linear operator $T^{{\varphi_1}}_{{\varphi_2}}: \clq_{\varphi_1} \raro \clq_{\varphi_1}$ defined by
\[
T^{{\varphi_1}}_{{\varphi_2}}f=P_{\clq_{{\varphi_1}}} ({\varphi_2} f),
\]
for all $f\in \clq_{{\varphi_1}}$. In particular, if ${\varphi_2}\in H^{\infty}(\D^n)$, then it is easy to see that
\[
[T^{{\varphi_1}}_{{\varphi_2}}, T^{{\varphi_1}}_{z_j}] = 0,
\]
for all $j=1, \ldots, n$. It is important to observe that
\[
T^{{\varphi_1}}_{z_j} = P_{\clq_{{\varphi_1}}}M_{z_j}|_{\clq_{{\varphi_1}}} \qquad (j=1, \ldots, n),
\]
and the $n$-tuple $(T^{{\varphi_1}}_{z_1}, \ldots, T^{{\varphi_1}}_{z_n})$ is the \textit{model operator} (or the tuple of \textit{compressed
shifts}) on the model space $\clq_{\varphi_1}$. The following concrete question is of particular interest:

\begin{Question}\label{Q: 1}
Characterize symbols ${\varphi_2} \in L^{\infty}(\T^n)$ and inner functions ${\varphi_1} \in H^{\infty}(\D^n)$ such that $T^{{\varphi_1}}_{{\varphi_2}}$ is a partial isometry.
\end{Question}

Despite its naturalness, the roots of this question may be traced back to Brown and Douglas's classic study on single variable partially isometric Toeplitz operators \cite{BH} and subsequently to the recent paper on the polydisc \cite{DDJ}. For truncated Toeplitz operators with inner symbols, we have the following solution to the above problem:

\begin{thm}\label{thm intro TTO}
Let $n > 1$, and let ${\varphi_1}, {\varphi_2} \in H^{\infty}(\D^n)$ be non-constant inner functions. Then $T^{{\varphi_1}}_{{\varphi_2}}$ is a partial isometry if and only if there exist inner functions $\tilde{{\varphi_1}}, \tilde{{\varphi_2}}, \psi \in H^{\infty}(\D^n)$ such that $\tilde{{\varphi_1}}$ and $\tilde{{\varphi_2}}$ are separated and
\[
{\varphi_1}=\tilde{{\varphi_1}} \psi \text{ and } {\varphi_2}=\tilde{{\varphi_2}} \psi.
\]
\end{thm}

The $n=1$ case is outlined in Corollary \ref{cor PI n=1}. We also derive some other consequences from the above theorem, such as the classification of isometric truncated Toeplitz operators, a duality between ${\varphi_1}$ and ${\varphi_2}$, and so on.

Now we briefly describe the development of truncated Toeplitz operators. Sarason's seminal paper \cite{DS} essentially embraced the concept of truncated Toeplitz operators. Sarason was, however, inspired by a number of influential works in Hilbert function space theory, such as the Sz.-Nagy and Foias analytic model theory \cite{Bs NF}, his own commutant lifting theorem \cite{DS 1}, the celebrated $H^\infty$ functional calculus for contractions \cite{Ber}, to mention a few. Even in such a short period of time, this concept has evolved significantly, and it is now a topic of its own.

Finally, some (historical as well as well-known) thoughts on invariant subspaces and model spaces of $H^2(\D^n)$, $n \geq 1$. Suppose $n = 1$. Then the celebrated Beurling theorem \cite{AB} says that a nonzero closed subspace $\cls \subseteq H^2(\D)$ is invariant under $M_z$ if and only if there exists an inner function ${\varphi} \in H^\infty(\D)$ such that
\[
\cls = {\varphi} H^2(\D).
\]
Consequently, $M_z^*$-invariant subspaces of $H^2(\D)$ are precisely the model spaces (see \eqref{eqn: QM})
\[
\clq_{\varphi} = H^2(\D) /{\varphi} H^2(\D),
\]
for inner functions ${\varphi} \in H^\infty(\D)$. In particular, shift invariant subspaces of $H^2(\D)$ are cyclic and generated by inner functions in $H^{\infty}(\D)$. This view has been an important part of the theory of bounded linear operators and the theory of Hilbert function spaces. If $n > 1$, however, the structure of general shift-invariant subspaces becomes significantly more complex. In other words, the structure of joint invariant subspaces of $(M_{z_1}, \ldots, M_{z_n})$ is baffling, whereas model spaces and model operators are similarly puzzling whenever $n > 1$ (cf. \cite{GW, GGK}). Also see \cite{BCS} in the context of Clark theory on the bidisc. All in all, the operator theory and function theory on the polydisc is thought to be much more subtle and mysterious.

The remainder of the paper is structured as follows: Section \ref{sec: inner proj} is devoted to establishing the paper's central result. Section \ref{sec Douglas} answers Douglas's question, while Section \ref{sec TTO} gives classifications of partially isometric truncated Toeplitz operators. Section \ref{sec example} concludes with some illustrated examples based on the core results obtained in this paper.

\newsection{Pairs of inner projections}\label{sec: inner proj}

The goal of this section is to prove the central result of this paper. We start by reviewing the fundamentals of partial isometries on Hilbert spaces, which will be used throughout the course of the investigation.

Let $T$ be a bounded linear operator on $\clh$. We say that $T$ is a \textit{partial isometry} if $T|_{(\ker T)^\perp}$ is an isometry, that is
\[
\|Th\|=\|h\| \qquad (h\in (\ker T)^\perp).
\]
If $T$ is a partial isometry, then $(\ker T)^\perp$ is referred to as the \textit{initial space} of $T$ and $\text{ran} T$ as the \textit{final space}. The following partial isometry characterizations are well-known and easy to prove (see \cite[Proposition 4.38]{RGD}): Let $T$ be a bounded linear operator on $\clh$. The following are equivalent:

\begin{enumerate}
\item $T$ is a partial isometry.
\item $T^*$ is a partial isometry.
\item $TT^*T=T$.
\item $T^*T$ is a projection.
\item $TT^*$ is a projection.
\end{enumerate}

The classical result of Brown and Douglas \cite{BH} and its subsequent generalization to several variables \cite{DDJ} serve as the foundation for proving the main result of this section. We recollect the general version from \cite[Theorem 1.1]{DDJ} because it will be used throughout the work. Recall that the \textit{Toeplitz operator} with symbol $\vp \in L^\infty(\T^n)$ is the bounded linear operator $T_\vp$ on $H^2(\D^n)$ defined by
\[
T_{\vp} = P_{H^2(\D^n)} L_\vp|_{H^2(\D^n)},
\]
where $L_\vp$ denotes the Laurent operator on $L^2(\T^n)$. Recall that
\[
L_\vp f = \vp f,
\]
for all $f \in L^2(\T^n)$.

\begin{thm}\label{thm KD}
Let $\vp$ be a nonzero function in $L^\infty(\T^n)$. The following hold:
\begin{enumerate}
\item If $n=1$, then $T_\vp$ is a partial isometry if and only if $T_\vp$ is either an isometry, or a coisometry.
\item If $n > 1$, then $T_\vp$ is a partial isometry if and only if there exist separated inner functions $\vp_1$ and $\vp_2$ in $H^\infty(\D^n)$ such that
\[
T_\vp = T_{\vp_1}^* T_{\vp_2}.
\]
\end{enumerate}
\end{thm}

Take note of the variations between the $n=1$ and $n>1$ case. This distinction will also be reflected in the results of one versus more than one variable in what follows. First, we focus on inner functions in $H^\infty(\D)$.

\begin{thm}\label{thm 1 var proj}
Let $\vp_1,\vp_2\in H^{\infty}(\D)$ be nonconstant inner functions. The following are equivalent:
\begin{enumerate}
\item $[P_{\vp_1H^2(\D)}, P_{\vp_2H^2(\D)}] = 0$.
\item $\vp_1H^2(\D) \subseteq \vp_2H^2(\D)$ or $\vp_2 H^2(\D) \subseteq \vp_1 H^2(\D)$.
\item $\vp_1$ divides $\vp_2$ or $\vp_2$ divides $\vp_1$.
\end{enumerate}

\end{thm}
\begin{proof}
Given that $\vp_1$ and $\vp_2$ are both inner functions and in light of Douglas's range inclusion theorem, it suffices to prove that (1) and (3) are equivalent. However, a portion of the proof for part (2) will be also shown in the proof that follows. For notational simplicity, assume $\cls_j=\vp_j H^2(\D)$, $j=1,2$. Suppose $[P_{\cls_1}, P_{\cls_2}] = 0$, that is
\[
P_{\cls_1} P_{\cls_2} = P_{\cls_2} P_{\cls_1}.
\]
Since $P_{\cls_i} = M_{\vp_i} M_{\vp_i}^*$, $i=1,2$, it follows that
\begin{equation}\label{eqn: proj comm}
(M_{\vp_1}M_{\vp_1}^*) (M_{\vp_2} M_{\vp_2}^*) = (M_{\vp_2}M_{\vp_2}^*) (M_{\vp_1}M_{\vp_1}^*).
\end{equation}
Observe that
\[
\begin{split}
M_z^*(M_{\vp_1}^* M_{\vp_2})M_z & = M_{\vp_1}^* M_z^* M_z M_{\vp_2}
\\
& = M_{\vp_1}^* M_{\vp_2},
\end{split}
\]
that is, $M_{\vp_1}^* M_{\vp_2}$ is a Toeplitz operator on $H^2(\D)$. This can also be proved in a variety of ways, such as the following: Since $\vp_1, \vp_2 \in H^\infty(\D)$, we conclude that $\bar{\vp}_1 \vp_2 \in L^\infty(\T)$ and $M_{\vp_1}^* M_{\vp_2} = T_{\bar{\vp}_1 \vp_2}$, where $T_{\bar{\vp}_1 \vp_2}$ is the Toeplitz operator with symbol $\bar{\vp}_1 \vp_2$. Now we claim that $T: = M_{\vp_1}^* M_{\vp_2}$ is a partial isometry. Indeed, in view of \eqref{eqn: proj comm}, we have
\[
\begin{split}
TT^*T & = M_{\vp_1}^* (M_{\vp_2}M_{\vp_2}^*) (M_{\vp_1}M_{\vp_1}^*) M_{\vp_2}
\\
& = M_{\vp_1}^* (M_{\vp_1}M_{\vp_1}^*)(M_{\vp_2}M_{\vp_2}^*)M_{\vp_2}
\\
& = M_{\vp_1}^*M_{\vp_2}
\\
& = T,
\end{split}
\]
which proves the claim. By part (1) of Theorem \ref{thm KD} (that is, the Brown and Douglas theorem in \cite{BH}), there exists an inner function $\vp\in H^{\infty}$ such that $T = M_{\vp}$ or $T = M_{\vp}^*$. Suppose $T=M_{\vp}$. Then $M_{\vp_1}^* M_{\vp_2} = M_\vp$ in particular implies that $M_{\vp_1}^* M_{\vp_2}$ is an isometry, that is
\[
M_{\vp_2}^*M_{\vp_1}M_{\vp_1}^* M_{\vp_2} = I.
\]
Multiplying by $ M_{\vp_2}$ from the left and $M_{\vp_2}^*$ from the right, we get
\[
M_{\vp_2}M_{\vp_2}^* M_{\vp_1}  M_{\vp_1}^* M_{\vp_2}M_{\vp_2}^* =M_{\vp_2}M_{\vp_2}^*,
\]
that is
\[
P_{\cls_2}P_{\cls_1}P_{\cls_2} = P_{\cls_2}.
\]
Since $P_{\cls_1}$ commutes with $P_{\cls_2}$, by our assumption, it follows that
\[
P_{\cls_2}=P_{\cls_1}P_{\cls_2},
\]
and hence
\[
\cls_2 = \vp_2 H^2(\D) \subseteq \cls_1 = \vp_1 H^2(\D).
\]
By Douglas's range inclusion theorem, we conclude that $\vp_1$ divides $\vp_2$ (also see \cite[page 20, Lemma 2.1]{Ber} ). Similarly, if $T$ is a coisometry, then $\vp_2$ devides $\vp_1$.

\NI For the converse direction, suppose that $\vp_1$ divides $\vp_2$. Equivalently, there exists an inner function $\vp\in H^{\infty}(\D)$ such that $\vp_2=\vp\vp_1$. Then
\[
\begin{split}
P_{\cls_1}P_{\cls_2} & = M_{\vp_1}M_{\vp_1}^*M_{\vp_2}M_{\vp_2}^*
\\
& = M_{\vp_1} M_{\vp_1}^* M_{\vp_1}M_{\vp}M_{\vp_1}^*M_{\vp}^*
\\
& = M_{\vp_1}M_{\vp}M_{\vp_1}^*M_{\vp}^*
\\
& = M_{\vp_2}M_{\vp_2}^*
\\
& =P_{\cls_2}.
\end{split}
\]
In particular
\[
\begin{split}
P_{\cls_2}P_{\cls_1} & = P_{\cls_1} P_{\cls_2} P_{\cls_1}
\\
& = P_{\cls_1} (P_{\cls_1} P_{\cls_2})^*
\\
& = P_{\cls_1} (P_{\cls_2})^*
\\
& = P_{\cls_1} P_{\cls_2}.
\end{split}
\]
The same conclusion holds if $\vp_2$ divides $\vp_1$. This completes the proof of the theorem.
\end{proof}

Now we turn to the case of several variables. Let us record the following simple fact (see the equality (3.3) in \cite{DDJ}): Let $\vp_1$ and $\vp_2$ be separated functions in $H^\infty(\D^n)$. Then
\begin{equation}\label{eqn: dc of TO}
M_{\vp_1}^* M_{\vp_2} = M_{\vp_2} M_{\vp_1}^*.
\end{equation}

\begin{thm}\label{thm n var proj}
Let $n > 1$, and let $\vp_1,\vp_2\in H^{\infty}(\D^n)$ be nonconstant inner functions. Then
\[
[P_{\vp_1H^2(\D^n)}, P_{\vp_2 H^2(\D^n)}] = 0,
\]
if and only if there exist inner functions $\tilde{\vp}_1, \tilde{\vp}_2, \psi \in H^\infty(\D^n)$ such that $\tilde{\vp}_1$ and $\tilde{\vp}_2$ are separated, and such that 
\[
\vp_1= \tilde{\vp}_1 \psi \text{ and } \vp_2 = \tilde{\vp}_2 \psi.
\]
\end{thm}

\begin{proof}
Set $\cls_j = \vp_jH^2(\D^n)$, $j=1,2$, and suppose $P_{\cls_1} P_{\cls_2} = P_{\cls_2} P_{\cls_1}$. We define
\[
X=M_{\vp_1}^* M_{\vp_2}.
\]
As in the proof of Theorem \ref{thm 1 var proj}, for each $i=1, \ldots, n$, we have
\[
\begin{split}
M_{z_i}^* X M_{z_i} = X.
\end{split}
\]
Therefore, by the algebraic (Brown and Halmos type) characterizations of Toeplitz operators \cite{MSS}, we conclude that $X$ is a Toeplitz operator. Furthermore, using the commutativity of projections, as in the proof of Theorem \ref{thm 1 var proj}, we again conclude that
\[
XX^*X = X,
\]
and hence $X$ is a partially isometric Toeplitz operator on $H^2(\D^n)$. By part (2) of Theorem \ref{thm KD}, there exist separated inner functions $\tilde{\vp}_1$ and $\tilde{\vp}_2$ in $H^{\infty}(\D^n)$ such that $X = M_{\tilde{\vp}_1}^*M_{\tilde{\vp}_2}$, that is
\[
M_{\vp_1}^*M_{\vp_2} = M_{\tilde{\vp}_1}^*M_{\tilde{\vp}_2}.
\]
This factorization will allow us to determine the initial and final spaces of the partial isometry $M_{\vp_1}^*M_{\vp_2}$. We use \eqref{eqn: dc of TO} and the above identity to compute
\[
\begin{split}
\ker (M_{\vp_1}^*M_{\vp_2}) & = \ker (M_{\tilde{\vp}_1}^*M_{\tilde{\vp}_2})
\\
& = \ker (M_{\tilde{\vp}_2}^*M_{\tilde{\vp}_1} M_{\tilde{\vp}_1}^*M_{\tilde{\vp}_2})
\\
& = \ker (M_{\tilde{\vp}_1} M_{\tilde{\vp}_1}^* M_{\tilde{\vp}_2}^* M_{\tilde{\vp}_2})
\\
& = \ker (M_{\tilde{\vp}_1} M_{\tilde{\vp}_1}^*).
\end{split}
\]
By the fact that $M_{\vp_1}^* M_{\vp_2}$ is a partial isometry, we immediately conclude that $\text{ran} (M_{\vp_2}^* M_{\vp_1})$ and $\text{ran} (M_{\vp_1}^* M_{\vp_2})$ are closed subspaces. Hence
\[
\begin{split}
\text{ran}(M_{\vp_2}^* M_{\vp_1}) & = (\ker (M_{\vp_1}^*M_{\vp_2}))^\perp
\\
& = (\ker M_{\tilde{\vp}_1}^*)^\perp
\\
& = \text{ran} M_{\tilde{\vp}_1}
\\
& = \tilde{\vp}_1 H^2(\D^n).
\end{split}
\]
Similarly
\[
\begin{split}
\text{ran} (M_{{\vp}_1}^* M_{{\vp}_2}) & = \text{ran} (M_{\tilde{\vp}_1}^* M_{\tilde{\vp}_2})
\\
& = \text{ran} (M_{\tilde{\vp}_1}^* M_{\tilde{\vp}_2} M_{\tilde{\vp}_2}^* M_{\tilde{\vp}_1})
\\
& = \text{ran} (M_{\tilde{\vp}_2} M_{\tilde{\vp}_2}^*)
\\
& = \tilde{\vp}_2 H^2(\D^n).
\end{split}
\]
Consequently, the initial and the final spaces of the partial isometry $M_{\vp_1}^*M_{\vp_2}$ are $\tilde{\vp}_1 H^2(\D^n)$ and $\tilde{\vp}_2 H^2(\D^n)$, respectively. On the other hand
\[
\vp_1 H^2(\D^n) = M_{\vp_2}^* M_{\vp_1}(\vp_2 H^2(\D^n)),
\]
and
\[
\vp_2 H^2(\D^n) = M_{\vp_1}^* M_{\vp_2}(\vp_1 H^2(\D^n)),
\]
imply that
\[
\vp_1 H^2(\D^n)\subseteq \text{ran}(M_{\vp_2}^* M_{\vp_1}),
\]
and
\[
\vp_2H^2(\D^n)\subseteq \text{ran}(M_{\vp_1}^* M_{\vp_2}).
\]
Therefore
\[
\vp_1 H^2(\D^n) \subseteq \tilde{\vp}_1 H^2(\D^n),
\]
and
\[
\vp_2 H^2(\D^n) \subseteq \tilde{\vp}_2 H^2(\D^n).
\]
In particular, there exist inner functions $\psi_1$ and $\psi_2$ in $H^{\infty}(\D^n)$ such that
\[
\vp_j = \tilde{\vp}_j \psi_j \qquad (j=1,2).
\]
Observe that by multiplying by $M_{\tilde{\vp}_2}^*$ from the left and $M_{\tilde{\vp}_1}$ from the right side of $M_{\vp_1}^*M_{\vp_2} = M_{\tilde{\vp}_1}^*M_{\tilde{\vp}_2}$ we get
\[
M_{\tilde{\vp}_2}^*M_{\vp_1}^*M_{\vp_2}M_{\tilde{\vp}_1}=I,
\]
that is, $M_{\vp_2}M_{\tilde{\vp}_1} = M_{\tilde{\vp}_2}M_{\vp_1}$, and hence
\[
\tilde{\vp}_1 \vp_2 = \tilde{\vp}_2 \vp_1.
\]
Then, in view of $\vp_j = \tilde{\vp}_j \psi_j$, $j=1,2$, we have
\[
\tilde{\vp}_1 \tilde{\vp}_2 \psi_2 = \tilde{\vp}_2 \tilde{\vp}_1 \psi_1,
\]
which yields that $\psi_1 = \psi_2$.

\NI We now turn to the converse. Suppose there exist inner functions $\tilde{\vp}_1, \tilde{\vp}_2, \psi \in H^\infty(\D^n)$ such that $\tilde{\vp}_1$ and $\tilde{\vp}_2$ are separated and $\vp_1 = \tilde{\vp}_1 \psi$ and $\vp_2 = \tilde{\vp}_2 \psi$. Then, in view of $P_{\cls_i} = M_{\vp_i} M_{\vp_i}^*$, $i=1,2$, and the identity \eqref{eqn: dc of TO} we compute
\[
\begin{split}
P_{\cls_1}P_{\cls_2} & = (M_{\tilde{\vp}_1} M_{\psi} M_{\tilde{\vp}_1}^* M_{\psi}^*) (M_{\tilde{\vp}_2} M_\psi M_{\tilde{\vp}_2}^* M_{\psi}^*)
\\
& = (M_{\tilde{\vp}_1} M_{\psi} M_{\tilde{\vp}_1}^*) (M_{\tilde{\vp}_2}M_{\tilde{\vp}_2}^*M_{\psi}^*)
\\
& = (M_{\tilde{\vp}_1} M_{\psi} M_{\tilde{\vp}_2}) (M_{\tilde{\vp}_1}^* M_{\tilde{\vp}_2}^*M_{\psi}^*)
\\
& = (M_{\tilde{\vp}_2} M_{\psi}  M_{\tilde{\vp}_2}^*) (M_{\tilde{\vp}_1} M_{\tilde{\vp}_1}^*M_{\psi}^*)
\\
& = (M_{\tilde{\vp}_2} M_{\psi}  M_{\tilde{\vp}_2}^* M_{\psi}^*) (M_{\psi} M_{\tilde{\vp}_1} M_{\tilde{\vp}_1}^*M_{\psi}^*)
\\
& = (M_{{\vp}_2} M_{{\vp}_2}^*) (M_{{\vp}_1} M_{{\vp}_1}^*).
\end{split}
\]
Therefore, $P_{\cls_1}P_{\cls_2} = P_{\cls_2}P_{\cls_1}$, which completes the proof of the theorem.
\end{proof}

It is particularly interesting to observe that in the preceding proof, we also show that if
\[
[P_{\vp_1H^2(\D^n)}, P_{\vp_2 H^2(\D^n)}] = 0,
\]
then
\[
P_{\vp_1H^2(\D^n)} P_{\vp_2 H^2(\D^n)} = M_{\psi} P_{\tilde{\vp}_1 \tilde{\vp}_2 H^2(\D^n)} M_{\psi}^*,
\]
or equivalently
\[
\vp_1H^2(\D^n) \cap \vp_2 H^2(\D^n) = \psi \tilde{\vp}_1 \tilde{\vp}_2 H^2(\D^n).
\]

\section{An answer to Douglas's question}\label{sec Douglas}

This section's goal is to answer the question of Douglas that was imposed in Problem \ref{Q: 3} (also see Bickel and Liaw \cite[page 104]{BL} for the origin of this problem): For which inner functions $\vp_1$ and $\vp_2$ in $H^{\infty}(\D^2)$ is the projection $P_{\clq_{\vp_1}} P_{\clq_{\vp_2}}$ finite rank?

We address this problem in a more general setting, that is, for any $n > 1$. In what follows, an invariant subspace of $H^2(\D^n)$ is one that is closed and invariant under $M_{z_i}$ for all $i=1, \ldots, n$, whereas a co-invariant subspace of $H^2(\D^n)$ means that the orthocomplement of that closed subspace is an invariant subspace of $H^2(\D^n)$.

We must recall a classical Ahern and Clark result (note the final corollary in \cite{AC}): Let $n > 1$ and let $\{f_i\}_{i=1}^m$ be a subset of $H^2(\D^n)$. Suppose $m < n$. Then the invariant subspace $\cls$ generated by $\{f_i\}_{i=1}^m$ is either all of $H^2(\D^n)$ or has infinite codimension.

\NI Recall that for an inner function $\vp \in H^\infty(\D^n)$, the model space $\clq_{\vp}$ is defined by (see \eqref{eqn: QM})
\[
\clq_{\vp} = H^2(\D^n)/\vp H^2(\D^n).
\]
The following is now easy:

\begin{cor}\label{cor inf dim QM}
Let $n > 1$ and let $\vp \in H^\infty(\D^n)$ be a nonconstant inner function. Then
\[
\text{dim} \clq_\vp = \infty.
\]
\end{cor}
\begin{proof}
Observe that $\clq_\vp^\perp = \vp H^2(\D^n)$. In particular, $\clq_\vp^\perp$ is a cyclic invariant subspace of $H^2(\D^n)$. The result now follows from Ahern and Clark.
\end{proof}

We are now ready to face Problem \ref{Q: 3} for any $n > 1$:

\begin{thm}\label{thm Douglas}
Let $\vp_1$ and $\vp_2$ be inner functions in $H^{\infty}(\D^n)$, $n \geq 2$. Then the following hold:
\begin{enumerate}
\item If $n=2$, then $P_{\clq_{\vp_1}} P_{\clq_{\vp_2}}$ is a finite rank projection if and only if either one of the following conditions holds:
\begin{enumerate}
\item $\vp_1$ or $\vp_2$ is a constant function;
\item $\vp_1$ and $\vp_2$ are separated and finite Blaschke products.
\end{enumerate}
\item If $n > 2$, then it is impossible for $P_{\clq_{\vp_1}} P_{\clq_{\vp_2}}$ to be a finite rank projection unless $P_{\clq_{\vp_1}} P_{\clq_{\vp_2}} = 0$.
\end{enumerate}
\end{thm}
\begin{proof}
We know that
\[
P_{\clq_{\vp_i}} = I - P_{\vp_i H^2(\D^n)} \qquad (i=1,2).
\]
It is now easy to see that $P_{\clq_{\vp_1}}P_{\clq_{\vp_2}}$ is a projection if and only if $P_{\vp_1H^2(\D^n)}P_{\vp_2H^2(\D^n)}$ is a projection, or equivalently (see the second paragraph in Section \ref{sec: 1})
\[
[P_{\vp_1H^2(\D^n)}, P_{\vp_2H^2(\D^n)}] = 0.
\]
This identity is automatically true if either $\vp_1$ or $\vp_2$ is a constant function (see the paragraph following Theorem \ref{thm intro inner}). Therefore, we assume that $\vp_1$ and $\vp_2$ are non-constant functions. Again, $P_{\clq_{\vp_1}}P_{\clq_{\vp_2}}$ is a projection if and only if
\[
[P_{\vp_1H^2(\D^n)}, P_{\vp_2H^2(\D^n)}] = 0.
\]
Equivalently, by Theorem \ref{thm n var proj}, there exist inner functions $\tilde{\vp}_1, \tilde{\vp}_2,\psi \in H^{\infty}(\D^n)$ such that
$\tilde{\vp}_1$ and $\tilde{\vp}_2$ are separated and
\[
\vp_1=\tilde{\vp}_1 \psi \text{ and } \vp_2 = \tilde{\vp}_2 \psi.
\]
Let us pause for a moment and simplify things. If the foregoing identities are true, then we see that
\[
I - M_{\vp_i} M_{\vp_i}^* = (I - M_\psi M_{\psi}^*) \oplus M_{\psi} (I - M_{\tilde{\vp_i}} M_{\tilde{\vp_i}}^*) M_{\psi}^*,
\]
for all $i=1,2$. Since $M_{\vp_i}$ and $M_{\tilde{\vp_i}}$, $i=1,2$, and $M_\psi$ are isometric operators acting on $H^2(\D^n)$, it follows that
\[
\text{ran} (I - M_{\vp_i} M_{\vp_i}^*) = \clq_{\vp_i} \text{ and } \text{ran}(I - M_\psi M_{\psi}^*) = \clq_{\psi},
\]
and
\[
\text{ran} (M_\psi(I - M_{\tilde{\vp_i}} M_{\tilde{\vp_i}}^*) M_{\psi}^*) = \psi \clq_{\tilde{\vp_i}},
\]
which yields
\[
\clq_{\vp_i} = \clq_{\psi} \oplus \psi \clq_{\tilde{\vp_i}},
\]
for all $i=1,2$. Clearly
\[
\clq_{\vp_1} \cap \clq_{\vp_2} = \clq_{\psi} \oplus \psi (\clq_{\tilde{\vp_1}} \cap \clq_{\tilde{\vp_2}}).
\]
After this simplification, now we return to the main body of the proof and define the bounded linear operator $P$ on $H^2(\D^n)$ by
\[
P = P_{\clq_{\vp_1}}P_{\clq_{\vp_2}}.
\]
If we know that $P$ is a projection, then
\[
\text{ran} P = \clq_{\vp_1}\cap \clq_{\vp_2},
\]
and consequently
\begin{equation}\label{eqn ran p decom}
\text{ran} P = \clq_{\psi} \oplus \psi (\clq_{\tilde{\vp_1}} \cap \clq_{\tilde{\vp_2}}).
\end{equation}
Therefore, $P$ is a finite rank projection if and only if \eqref{eqn ran p decom} holds and
\[
\text{dim}(\text{ran} P ) < \infty;
\]
or equivalently, \eqref{eqn ran p decom} holds and
\[
\text{dim}(\clq_{\psi}) < \infty,
\]
and (recall that $\psi$ is inner and hence $M_\psi$ is an isometry)
\begin{equation}\label{eqn dim of Q}
\text{dim}(\clq_{\tilde{\vp_1}} \cap \clq_{\tilde{\vp_2}}) < \infty.
\end{equation}
We now divide the remaining proof into the two cases of $n=2$ and $n>2$. However, the following two remarks must be made before we discuss these cases separately: According to Corollary \ref{cor inf dim QM}, the assumption $\text{dim}(\clq_{\psi}) < \infty$ compels the inner function $\psi$ to be a unimodular constant, that is, there exists $\alpha \in \T$ such that
\[
\psi \equiv \alpha,
\]
which also forces
\[
\clq_{\psi} = \{0\}.
\]
Secondly, recall that in the case of $n=1$, a co-invariant subspace $\clq_{\vp}\subseteq H^2(\D)$ is finite-dimensional if and only if $\vp$ is a finite Blaschke product, and, in this case
\[
\text{rank} P_{\clq_{\vp}} = \# \{z \in \D: \vp(z) = 0\},
\]
counting multiplicities.

\NI \textsf{Case 1:} Suppose $n=2$. Then $P$ is a finite rank projection if and only if there exist separated inner functions $\tilde{\vp}_1, \tilde{\vp}_2 \in H^\infty(\D^2)$ and a scalar $\alpha \in \T$ such that
\[
\vp_1 = \alpha \tilde{\vp}_1 \text{ and } \vp_2 = \alpha \tilde{\vp}_2,
\]
and \eqref{eqn dim of Q} holds. In view of this, we may state the following (as one may simply replace $\vp_i$ by $\alpha \tilde{\vp}_i$, $i=1,2$): $P$ is a finite rank projection if and only if $\vp_1$ depends only on the first variable and $\vp_2$ depends only on the second, and
\[
\text{dim}(\clq_{{\vp_1}} \cap \clq_{{\vp_2}}) < \infty.
\]
Now we analyze the co-invariant subspace $\clq_{{\vp_1}} \cap \clq_{{\vp_2}}$ under the assumption that $\vp_1$ depends only on $z_1$ and $\vp_2$ depends only on $z_2$. There exist functions $\sigma_1, \sigma_2 \in H^\infty(\D)$ such that
\[
\vp_1(z_1, z_2) = \sigma_1(z_1) \text{ and } \vp_2(z_1, z_2) = \sigma_2(z_2),
\]
for all $(z_1, z_2) \in \D^2$. We avoid distinguishing $H^2(\D^2)$ and $H^2(\D) \otimes H^2(\D)$ in the following because of the inherent unitary equivalence. For instance, $M_{\vp_1} = M_{\sigma_1} \otimes I_{H^2(\D)}$. We compute
\[
\begin{split}
\clq_{\vp_1} & = (I_{H^2(\D^2)} - M_{\vp_1} M_{\vp_1}^*) H^2(\D^2)
\\
& = ((I_{H^2(\D)} - M_{\sigma_1} M_{\sigma_1}^*) \otimes I_{H^2(\D)}) (H^2(\D) \otimes H^2(\D))
\\
& = ((I_{H^2(\D)} - M_{\sigma_1} M_{\sigma_1}^*) H^2(\D)) \otimes H^2(\D)
\\
& = \clq_{\sigma_1} \otimes H^2(\D),
\end{split}
\]
and similarly
\[
\clq_{\vp_2} = H^2(\D) \otimes \clq_{\sigma_2}.
\]
Therefore, we conclude that
\[
\clq_{{\vp_1}} \cap \clq_{{\vp_2}} = \clq_{\sigma_1} \otimes \clq_{\sigma_2}.
\]
Then
\[
\text{dim}(\clq_{{\vp_1}} \cap \clq_{{\vp_2}}) = \text{dim} \clq_{\sigma_1} \times \text{dim} \clq_{\sigma_2},
\]
and hence
\[
\text{dim}(\clq_{{\vp_1}} \cap \clq_{{\vp_2}})< \infty,
\]
if and only if
\[
\text{dim} \clq_{\sigma_1}, \text{dim} \clq_{\sigma_2} < \infty,
\]
or equivalently, $\sigma_1$ and $\sigma_2$ are finite Blaschke products. In summary, $P_{\clq_{\vp_1}} P_{\clq_{\vp_2}}$ is a finite rank projection if and only if $\vp_1$ and $\vp_2$ are separated and finite Blaschke products.

\NI \textsf{Case 2:} Suppose $n > 2$. As in Case 1, it again follows that $\psi$ is a unimodular constant, and hence: $P$ is a finite rank projection if and only if $\vp_1$ and $\vp_2$ are separated, and
\[
\text{dim}(\clq_{{\vp_1}} \cap \clq_{{\vp_2}}) < \infty.
\]
On the other hand
\[
(\clq_{{\vp_1}} \cap \clq_{{\vp_2}})^\perp = \overline{\text{span}}\{\vp_1 H^2(\D^n), \vp_2 H^2(\D^n)\},
\]
implies that the invariant subspace $(\clq_{{\vp_1}} \cap \clq_{{\vp_2}})^\perp$ is generated by at most two inner functions (namely $\vp_1$ and $\vp_2$). Since $n \geq 3$, by Ahern and Clark, it follows that
\[
\overline{\text{span}}\{\vp_1 H^2(\D^n), \vp_2 H^2(\D^n)\} = H^2(\D^n).
\]
Therefore
\[
\clq_{{\vp_1}} \cap \clq_{{\vp_2}} = \{0\}.
\]
Since $[P_{\clq_{\vp_1}}, P_{\clq_{\vp_2}} ] = 0$, it follows that $P_{\clq_{\vp_1}} P_{\clq_{\vp_2}} = 0$. Therefore, if $P_{\clq_{\vp_1}} P_{\clq_{\vp_2}} \neq 0$, then $P_{\clq_{\vp_1}} P_{\clq_{\vp_2}}$ cannot be a finite rank projection.
\end{proof}

If $f$ is a complex-valued function on $\Omega \subseteq \mathbb{C}^n$, we denote by $\clz(f)$ the zero set of $f$. From the above proof, it follows that if $P_{\clq_{\vp_1}} P_{\clq_{\vp_2}}$ is a finite rank projection (of course, on $H^2(\D^2)$), then
\[
\text{rank} (P_{\clq_{\vp_1}} P_{\clq_{\vp_2}}) = \# \clz(\vp_1) \times \# \clz(\vp_2).
\]

The question posed by Douglas also makes sense in one variable situation. We claim the following: Suppose $\vp_1$ and $\vp_2$ be nonconstant inner functions in $H^\infty(\D)$. Then $P_{\clq_{\vp_1}} P_{\clq_{\vp_2}}$ is a finite rank projection if and only if either $\vp_1$ is a finite Blaschke product and divides $\vp_2$ or $\vp_2$ is a finite Blaschke product and divides $\vp_1$.

\NI The proof proceeds as follows: As in the preceding theorem, $P_{\clq_{\vp_1}} P_{\clq_{\vp_2}}$ is a projection if and only if $[P_{\vp_1 H^2(\D)}, P_{\vp_2 H^2(\D)}] = 0$, which, by Theorem \ref{thm 1 var proj}, is equivalent to the condition that $\vp_1$ divides $\vp_2$ or $\vp_2$ divides $\vp_1$. Suppose $\vp_1$ divides $\vp_2$, that is, $\vp_2 = \psi \vp_1$ for some inner function $\psi \in H^\infty(\D)$. This is equivalent to the condition that
\[
\clq_{\vp_1} \subseteq \clq_{\vp_2}.
\]
In this case, we have $\clq_{\vp_1} \cap \clq_{\vp_2} = \clq_{\vp_1}$, which implies that
\[
P_{\clq_{\vp_1}} P_{\clq_{\vp_2}} = P_{\clq_{\vp_1}},
\]
and consequently, $P_{\clq_{\vp_1}} P_{\clq_{\vp_2}}$ is a finite rank operator if and only if $\vp_1$ is a finite Blaschke product. If $\vp_2$ divides $\vp_1$, then the proof follows similarly. This completes the proof of the claim.

\newsection{Truncated Toeplitz operators}\label{sec TTO}

In this section we characterize partially isometric truncated Toeplitz operators with inner symbols. Throughout what follows ${{\varphi_1}}$ and ${\varphi_2}$ will be nonconstant inner functions in $H^{\infty}(\D^n)$. Recall that the truncated Toeplitz operator on $\clq_{\varphi_1}$ with symbol ${{\varphi_2}}$ is the bounded linear operator $T_{{{\varphi_2}}}^{{\varphi_1}}$ on $\clq_{\varphi_1}$, where
\[
T_{{{\varphi_2}}}^{{\varphi_1}} = P_{\clq_{\varphi_1}} M_{{{\varphi_2}}}|_{\clq_{\varphi_1}}.
\]
The following lemma will turn out to be very useful in what follows.

\begin{lem}\label{lem TTO}
Let ${{\varphi}}$ and ${\varphi_2}$ be nonconstant inner functions in $H^{\infty}(\D^n)$, $n \geq 1$. Suppose $T_{{{\varphi_2}}}^{{\varphi_1}}$ is a partial isometry. The following hold:
\begin{enumerate}
\item $M_{{{\varphi_2}}}^* {\varphi_1} \clq_{{\varphi_2}} \subseteq \clq_{\varphi_1}$.
\item $(\ker T_{{{\varphi_2}}}^{{\varphi_1}})^\perp = \clq_{\varphi_1} \ominus M_{{{\varphi_2}}}^* {\varphi_1} \clq_{{\varphi_2}}$.
\item $\ker T_{{{\varphi_2}}}^{{\varphi_1}} = \overline{M_{{{\varphi_2}}}^* {\varphi_1} \clq_{{\varphi_2}}}$.
\end{enumerate}
\end{lem}
\begin{proof}
Note that $\clq_{{\varphi_2}} = \ker M_{{\varphi_2}}^*$. For each $f \in \clq_{{\varphi_2}}$ and $g \in H^2(\D^n)$, we have
\[
\begin{split}
\langle M_{{{\varphi_2}}}^* {\varphi_1} f, {\varphi_1} g \rangle & = \langle M_{{\varphi_1}}^* M_{{{\varphi_2}}}^* {\varphi_1} f, g \rangle
\\
& = \langle M_{{{\varphi_2}}}^* f, g \rangle
\\
& = 0,
\end{split}
\]
which yields (1). For (2), we observe that $f \in (\ker T_{{{\varphi_2}}}^{{\varphi_1}})^\perp$ if and only if $\|T_{{{\varphi_2}}}^{{\varphi_1}} f\| = \|f\|$. We have on one hand $\|f\| = \|{{\varphi_2}} f\|$ (since ${{\varphi_2}}$ is inner), and on the other hand $T_{{{\varphi_2}}}^{{\varphi_1}} f = P_{\clq_{\varphi_1}} ({{\varphi_2}} f)$ for all $f \in \clq_{\varphi_1}$. Therefore, $f \in (\ker T_{{{\varphi_2}}}^{{\varphi_1}})^\perp$ if and only if
\[
\|P_{\clq_{\varphi_1}} ({{\varphi_2}} f)\| = \|{{\varphi_2}} f\|,
\]
or equivalently
\[
{{\varphi_2}} f \in \clq_{\varphi_1} = ({\varphi_1} H^2(\D^n))^\perp.
\]
Therefore
\[
(\ker T_{{{\varphi_2}}}^{{\varphi_1}})^\perp = \{f \in \clq_{\varphi_1}: f \perp M_{{{\varphi_2}}}^* {\varphi_1} H^2(\D^n)\}.
\]
Note that $M_{{{\varphi_2}}}^* {\varphi_1} {{\varphi_2}} H^2(\D^n) = {\varphi_1} H^2(\D^n)$, and $M_{{{\varphi_2}}}^* {\varphi_1} {{\varphi_2}} H^2(\D^n) \perp M_{{{\varphi_2}}}^* {\varphi_1} \clq_{{\varphi_2}}$. Indeed, to prove the latter claim, for each $f \in H^2(\D^n)$ and $g \in \clq_{{\varphi_2}}$, we compute
\[
\begin{split}
\langle M_{{{\varphi_2}}}^* {\varphi_1} {{\varphi_2}} f, M_{{{\varphi_2}}}^* {\varphi_1} g \rangle & = \langle {\varphi_1} f, M_{{{\varphi_2}}}^* {\varphi_1} g \rangle
\\
& = \langle M_{{{\varphi_2}}} f, g \rangle
\\
& = 0.
\end{split}
\]
This proves the claim. Therefore, writing $H^2(\D^n) = {{\varphi_2}} H^2(\D^n) \oplus \clq_{{\varphi_2}}$ we find
\[
\overline{M_{{{\varphi_2}}}^* {\varphi_1} H^2(\D^n)} = {\varphi_1} H^2(\D^n) \oplus \overline{M_{{{\varphi_2}}}^* {\varphi_1} \clq_{{\varphi_2}}}.
\]
Now $f \in \clq_{\varphi_1}$ automatically implies that $f \perp {\varphi_1} H^2(\D^n)$, and hence
\[
(\ker T_{{{\varphi_2}}}^{{\varphi_1}})^\perp = \clq_{\varphi_1} \ominus M_{{{\varphi_2}}}^* {\varphi_1} \clq_{{\varphi_2}},
\]
which completes the proof of part (2). Finally, (3) follows from both (1) and (2).
\end{proof}

We take a step towards the main
theorem of this section:

\begin{prop}\label{prop: TTO}
Let ${{\varphi_1}}$ and ${\varphi_2}$ be nonconstant inner functions in $H^{\infty}(\D^n)$. If $T_{{{\varphi_2}}}^{{\varphi_1}}$ is a partial isometry, then
\[
[P_{{\varphi_1} H^2(\D^n)}, P_{{{\varphi_2}} H^2(\D^n)}] = 0.
\]
\end{prop}
\begin{proof}
Suppose $f \in \clq_{\varphi_1}$. Then $T_{{{\varphi_2}}}^{{\varphi_1}} f = 0$ if and only if
\[
P_{\clq_{\varphi_1}} {{\varphi_2}} f = 0,
\]
or equivalently, $(I - M_{\varphi_1} M_{\varphi_1}^*) {{\varphi_2}} f = 0$. So we conclude that
\[
\ker T_{{{\varphi_2}}}^{{\varphi_1}} = \{f \in \clq_{\varphi_1}: M_{\varphi_1} M_{\varphi_1}^* {{\varphi_2}} f = {{\varphi_2}} f\}.
\]
Since part (3) of Lemma \ref{lem TTO} in particular implies that $M_{{{\varphi_2}}}^* {\varphi_1} \clq_{{\varphi_2}} \subseteq \ker T_{{{\varphi_2}}}^{{\varphi_1}}$, we conclude that
\[
M_{\varphi_1} M_{\varphi_1}^* {{\varphi_2}} (M_{{\varphi_2}}^* {\varphi_1} g) = {{\varphi_2}} (M_{{\varphi_2}}^* {\varphi_1} g),
\]
for all $g \in \clq_{{\varphi_2}}$. Then, writing
\[
{{\varphi_2}} (M_{{\varphi_2}}^* {\varphi_1} g) = (M_{{\varphi_2}} M_{{\varphi_2}}^*)(M_{\varphi_1} M_{\varphi_1}^*) {\varphi_1} g,
\]
into the above identity, we find
\[
P_{{\varphi_1} H^2(\D^n)} P_{{{\varphi_2}} H^2(\D^n)} ({\varphi_1} g) = P_{{{\varphi_2}} H^2(\D^n)} P_{{\varphi_1} H^2(\D^n)} ({\varphi_1} g),
\]
for all $g \in \clq_{{\varphi_2}}$, that is
\[
P_{{\varphi_1} H^2(\D^n)} P_{{{\varphi_2}} H^2(\D^n)} M_{\varphi_1}|_{\clq_{{\varphi_2}}} = P_{{{\varphi_2}} H^2(\D^n)} P_{{\varphi_1} H^2(\D^n)} M_{\varphi_1}|_{\clq_{{\varphi_2}}}.
\]
Now we claim that
\[
P_{{\varphi_1} H^2(\D^n)} P_{{{\varphi_2}} H^2(\D^n)} M_{\varphi_1}|_{{{\varphi_2}} H^2(\D^n)} = P_{{{\varphi_2}} H^2(\D^n)} P_{{\varphi_1} H^2(\D^n)} M_{\varphi_1}|_{{{\varphi_2}} H^2(\D^n)}.
\]
To this end, fix $f \in H^2(\D^n)$ and compute
\[
\begin{split}
(P_{{\varphi_1} H^2(\D^n)} P_{{{\varphi_2}} H^2(\D^n)} M_{\varphi_1}) ({{\varphi_2}} f) & = P_{{\varphi_1} H^2(\D^n)} P_{{{\varphi_2}} H^2(\D^n)} {{\varphi_2}} {\varphi_1} f
\\
& = P_{{\varphi_1} H^2(\D^n)} {\varphi_1} {{\varphi_2}} f
\\
& = {\varphi_1} {{\varphi_2}} f
\\
& = (P_{{{\varphi_2}} H^2(\D^n)} P_{{\varphi_1} H^2(\D^n)} M_{\varphi_1}) ({{\varphi_2}} f),
\end{split}
\]
which proves the claim. Therefore
\[
P_{{\varphi_1} H^2(\D^n)} P_{{{\varphi_2}} H^2(\D^n)} M_{\varphi_1} = P_{{{\varphi_2}} H^2(\D^n)} P_{{\varphi_1} H^2(\D^n)} M_{\varphi_1},
\]
and hence
\[
\begin{split}
P_{{\varphi_1} H^2(\D^n)} P_{{{\varphi_2}} H^2(\D^n)} P_{{\varphi_1} H^2(\D^n)} & = P_{{{\varphi_2}} H^2(\D^n)} P_{{\varphi_1} H^2(\D^n)} P_{{\varphi_1} H^2(\D^n)}
\\
& = P_{{{\varphi_2}} H^2(\D^n)} P_{{\varphi_1} H^2(\D^n)},
\end{split}
\]
and consequently, $[P_{{\varphi_1} H^2(\D^n)}, P_{{{\varphi_2}} H^2(\D^n)}] = 0$.
\end{proof}

Now we are ready for the classifications of partially isometric Truncated Toeplitz operators with inner symbols.

\begin{thm}\label{thm PI n>1}
Let $n > 1$. Let ${{\varphi_1}}$ and ${\varphi_2}$ be nonconstant inner functions in $H^{\infty}(\D^n)$. Then $T_{{{\varphi_2}}}^{{\varphi_1}}$ is a partial isometry if and only if there exist inner functions $\tilde{{{\varphi_1}}}, \tilde{{\varphi_2}}, \psi \in H^{\infty}(\D^n)$ such that	$\tilde{{{\varphi_1}}}$ and $\tilde{{\varphi_2}}$ are separated and
\[
{\varphi_1} = \tilde{{\varphi_1}} \psi \text{ and } {{\varphi_2}} = \tilde{{{\varphi_2}}} \psi.
\]
\end{thm}
\begin{proof}
Suppose $T_{{{\varphi_2}}}^{{\varphi_1}}$ is a partial isometry. By Proposition \ref{prop: TTO}, we know that
\[
[P_{{\varphi_1} H^2(\D^n)}, P_{{{\varphi_2}} H^2(\D^n)}] = 0.
\]
and the conclusion now follows directly from Theorem \ref{thm n var proj}. To show the reverse assume ${\varphi_1} = \tilde{{\varphi_1}} \psi$ and ${{\varphi_2}} = \tilde{{{\varphi_2}}} \psi$ for some inner function $\psi$ and separated inner functions $\tilde{{{\varphi_1}}}$ and $\tilde{{\varphi_2}}$ in $H^{\infty}(\D^n)$. Recall from \eqref{eqn: dc of TO} that $M_{\tilde{{{\varphi_2}}}}^* M_{\tilde{{\varphi_1}}} = M_{\tilde{{\varphi_1}}} M_{\tilde{{{\varphi_2}}}}^*$. Then, on one hand
\[
\begin{split}
(I - M_{\varphi_1} M_{\varphi_1}^*) M_{{\varphi_2}} & = (I - M_{\tilde{{\varphi_1}}} M_\psi M_{\tilde{{\varphi_1}}}^* M_\psi^*) M_{\tilde{{{\varphi_2}}}} M_\psi
\\
&= M_{\tilde{{{\varphi_2}}}} M_\psi - M_{\tilde{{\varphi_1}}} M_\psi M_{\tilde{{\varphi_1}}}^* M_{\tilde{{{\varphi_2}}}}
\\
& = M_{\tilde{{{\varphi_2}}}} M_\psi(I - M_{\tilde{{\varphi_1}}} M_{\tilde{{\varphi_1}}}^*),
\end{split}
\]
and on the other hand that
\[
\begin{split}
(I - M_{\tilde{{\varphi_1}}} M_{\tilde{{\varphi_1}}}^*) M_{\varphi_1} M_{\varphi_1}^* & = (I - M_{\tilde{{\varphi_1}}} M_{\tilde{{\varphi_1}}}^*) M_\psi M_{\tilde{{\varphi_1}}} M_{\tilde{{\varphi_1}}}^* M_\psi^*
\\
&= M_\psi M_{\tilde{{\varphi_1}}} M_{\tilde{{\varphi_1}}}^* M_\psi^* - M_{\tilde{{\varphi_1}}} M_\psi M_{\tilde{{\varphi_1}}}^* M_\psi^*
\\
& = 0.
\end{split}
\]
Therefore
\[
\begin{split}
P_{\clq_{{\varphi_1}}}M_{{{\varphi_2}}}P_{\clq_{{\varphi_1}}} &=(I-M_{{\varphi_1}}M_{{\varphi_1}}^*)M_{{{\varphi_2}}}(I-M_{{\varphi_1}}M_{{\varphi_1}}^*)
\\
&= M_{\tilde{{{\varphi_2}}}} M_\psi(I - M_{\tilde{{\varphi_1}}} M_{\tilde{{\varphi_1}}}^*)(I-M_{{\varphi_1}}M_{{\varphi_1}}^*)
\\
& = M_{\tilde{{{\varphi_2}}}} M_\psi(I - M_{\tilde{{\varphi_1}}} M_{\tilde{{\varphi_1}}}^*) - M_{\tilde{{{\varphi_2}}}} M_\psi ((I - M_{\tilde{{\varphi_1}}} M_{\tilde{{\varphi_1}}}^*) M_{\varphi_1} M_{\varphi_1}^*)
\\
& = M_{\tilde{{{\varphi_2}}}} M_\psi(I - M_{\tilde{{\varphi_1}}} M_{\tilde{{\varphi_1}}}^*)
\\
& = M_{{{\varphi_2}}}P_{\clq_{\tilde{{\varphi_1}}}},
\end{split}
\]
and finally
\[
\begin{split}
(M_{{{\varphi_2}}}P_{\clq_{\tilde{{\varphi_1}}}})^* (M_{{{\varphi_2}}}P_{\clq_{\tilde{{\varphi_1}}}}) & = P_{\clq_{\tilde{{\varphi_1}}}} M_{{{\varphi_2}}}^* M_{{{\varphi_2}}} P_{\clq_{\tilde{{\varphi_1}}}}
\\
&= P_{\clq_{\tilde{{\varphi_1}}}},
\end{split}
\]
implies that $P_{\clq_{{\varphi_1}}}M_{{{\varphi_2}}}P_{\clq_{{\varphi_1}}}$ is a partial isometry. Since $P_{\clq_{{\varphi_1}}}M_{{{\varphi_2}}}P_{\clq_{{\varphi_1}}}|_{\clq_{\varphi_1}} = T_{{\varphi_2}}^{\varphi_1}$, we conclude that $T_{{\varphi_2}}^{\varphi_1}$ is a partial isometry.
\end{proof}

As a consequence of Theorem \ref{thm 1 var proj}, a similar but much simpler computation yields the following one-variable result:

\begin{cor}\label{cor PI n=1}
Let ${{\varphi_1}}$ and ${\varphi_2}$ be nonconstant inner functions in $H^{\infty}(\D)$. Then $T^{{\varphi_1}}_{{{\varphi_2}}}$ is a partial isometry if and only if ${{\varphi_2}}$ divides ${\varphi_1}$ or ${\varphi_1}$ divides ${{\varphi_2}}$.
\end{cor}

A more thorough look at the statement of Theorem \ref{thm PI n>1} reveals that the roles of the inner functions ${{\varphi_1}}$ and ${\varphi_2}$ can be swapped. This resulted in the following observation:

\begin{cor}\label{cor: duality 1}
Let ${{\varphi_1}}$ and ${\varphi_2}$ be nonconstant inner functions in $H^{\infty}(\D^n)$. Then $T_{{{\varphi_2}}}^{{\varphi_1}}$ is a partial isometry on $\clq_{{\varphi_1}}$ if and only if $T_{{\varphi_1}}^{{{\varphi_2}}}$ is a partial isometry on $\clq_{{{\varphi_2}}}$.
\end{cor}

Observe that a partial isometry $T$ is an isometry if and only if $\ker T = \{0\}$. In the following, we classify isometric truncated Toeplitz operators with inner symbols.

\begin{thm}\label{thm isom TTO}
Suppose $n > 1$. Let ${{\varphi_1}}$ and ${\varphi_2}$ be nonconstant inner functions in $H^{\infty}(\D^n)$. Then $T^{{\varphi_1}}_{{{\varphi_2}}}$ is an isometry if and only if ${{\varphi_1}}$ and ${\varphi_2}$ are separated.
\end{thm}

\begin{proof}
Suppose that $T^{{\varphi_1}}_{{{\varphi_2}}}$ is an isometry. By applying Theorem \ref{thm PI n>1} to $T^{{\varphi_1}}_{{{\varphi_2}}}$, we get inner functions $\tilde{{{\varphi_1}}}, \tilde{{\varphi_2}}$ and $\psi$ in $H^{\infty}(\D^n)$ such that $\tilde{{{\varphi_1}}}$ and $\tilde{{\varphi_2}}$ are separated and ${\varphi_1} = \tilde{{\varphi_1}} \psi$ and ${{\varphi_2}} = \tilde{{{\varphi_2}}} \psi$. By Lemma \ref{lem TTO}, we also know that
\[
\ker T^{{\varphi_1}}_{{{\varphi_2}}} = \clq_{{\varphi_1}} \ominus M_{{{\varphi_2}}}^* \clq_{{\varphi_1}}.
\]
Since $\ker T^{{\varphi_1}}_{{{\varphi_2}}} = \{ 0\}$, it follows that
\[
\clq_{{\varphi_1}}= \overline{M_{{{\varphi_2}}}^*\clq_{{\varphi_1}}}.
\]
Using the factorization ${\varphi_1} = \tilde{{\varphi_1}} \psi$, we find
\[
I - M_{\varphi_1} M_{\varphi_1}^* = (I - M_\psi M_\psi^*) \oplus M_\psi (I - M_{\tilde{{\varphi_1}}} M_{\tilde{{\varphi_1}}}^*)M_\psi,
\]
and hence (see the proof of Theorem \ref{thm Douglas})
\[
\clq_{\varphi_1} = \clq_\psi \oplus \psi \clq_{\tilde{{\varphi_1}}}.
\]
The equality $M_\psi^* \clq_\psi = \{0\}$ shows furthermore that
\[
\overline{M_\psi^* \clq_{\varphi_1}} = \clq_{\tilde{{\varphi_1}}},
\]
and hence
\[
\begin{split}
\clq_{{\varphi_1}} & = \overline{M_{{{\varphi_2}}}^*\clq_{{\varphi_1}}}
\\
& = \overline{M_{\tilde{{{\varphi_2}}}}^*M_\psi^* \clq_{\varphi_1}}
\\
& = \overline{M_{\tilde{{{\varphi_2}}}}^* \clq_{\tilde{{\varphi_1}}}}.
\end{split}
\]
We now apply the identity \eqref{eqn: dc of TO} to the pair $(M_{\tilde{{{\varphi_2}}}}, M_{\tilde{{\varphi_1}}})$ and observe that
\[
M_{\tilde{{{\varphi_2}}}}^* (I - M_{\tilde{{\varphi_1}}} M_{\tilde{{\varphi_1}}}^*) = (I - M_{\tilde{{\varphi_1}}} M_{\tilde{{\varphi_1}}}^*) M_{\tilde{{{\varphi_2}}}}^*.
\]
Since $M_{\tilde{{{\varphi_2}}}}^*$ is onto (as it is a co-isometry), this yields
\[
\overline{M_{\tilde{{{\varphi_2}}}}^* \clq_{\tilde{{\varphi_1}}}} = \clq_{\tilde{{\varphi_1}}},
\]
and we conclude that
\[
\clq_{{\varphi_1}} = \clq_{\tilde{{\varphi_1}}},
\]
or equivalently, ${\varphi_1} H^2(\D^n)=\tilde{{\varphi_1}} H^2(\D^n)$. By the uniqueness part of the Beurling theorem, there is a unimodular constant $\alpha$ such that ${\varphi_1} = \alpha \tilde{{\varphi_1}}$, which in turn yields
\[
\psi = \alpha,
\]
and hence ${\varphi_1} = \alpha \tilde{{\varphi_1}}$ and ${{\varphi_2}} = \alpha \tilde{{{\varphi_2}}}$. Therefore, ${\varphi_1}$ and $\varphi_{2}$ are separated. To prove the converse, we assume that ${\varphi_1}$ and $\varphi_{2}$ are separated. For simplicity, assume that ${{\varphi_2}}$ depends only on the first $m$ variables and ${\varphi_1}$ depends only on the remaining $n-m$ variables. It is now easy to see that
\[
T^{{\varphi_1}}_{{{\varphi_2}}} = M_{{{\varphi_2}}}\otimes I \text{ on } H^2(\D^m) \otimes (H^2(\D^{n-m})/ {\varphi_1} H^2(\D^{n-m})).
\]
This leads naturally to the conclusion that $T^{{\varphi_1}}_{{{\varphi_2}}}$ is an isometry.
\end{proof}

A careful inspection of the aforementioned proof shows there are no nontrivial isometric truncated Toeplitz operators in the $n=1$ case. Moreover, as in Corollary \ref{cor: duality 1}, we have the following consequence of the above theorem:

\begin{cor}\label{cor: duality 2}
Suppose $n>1$. Let ${{\varphi_1}}$ and ${\varphi_2}$ be nonconstant inner functions in $H^{\infty}(\D^n)$. Then $T_{{{\varphi_2}}}^{{\varphi_1}}$ is an isometry on $\clq_{{\varphi_1}}$ if and only if $T_{{\varphi_1}}^{{{\varphi_2}}}$ is an isometry on $\clq_{{{\varphi_2}}}$.
\end{cor}

By comparing Theorems \ref{thm PI n>1} and \ref{thm isom TTO}, we conclude that the existence of the nontrivial inner function $\psi$ in $H^\infty(\D^n)$, $n > 1$, is the primary distinction between isometric and partial isometric truncated Toeplitz operators. Finally, it would be interesting to explore whether there are any links between the concept of partial orders of partial isometries in \cite{Garcia} and the theory of partially isometric Toeplitz operators or partially isometric truncated Toeplitz operators on $\D^n$, $n \geq 1$.

\section{Examples}\label{sec example}

Let us wrap up this paper with two simple examples. Given $\alpha \in \D$, denote by $b_\alpha$ the Blaschke factor corresponding to $\alpha$:
\[
b_{\alpha}(z) = \frac{z - \alpha}{1 - \bar{\alpha} z} \qquad (z \in \D).
\]
Finite and infinite Blaschke products (as long as the associated sequences from $\D$ satisfy the Blaschke condition) are well-known examples of inner functions in $H^\infty(\D)$. Also, for each $\alpha \in \D$, define the Szeg\"{o} kernel function $\bS(\cdot, \alpha)$ on $\D$ by
\[
(\bS(\cdot, \alpha))(z) = \frac{1}{1 - \bar{\alpha} z} \qquad (z \in \D).
\]
Finally, recall that for each $\vp \in H^\infty(\D)$, we have
\[
M_{\vp}^* \bS(\cdot, \alpha) = \overline{\vp(\alpha)} \bS(\cdot, \alpha).
\]
The following example validates Theorem \ref{thm 1 var proj}.

\begin{ex}
Let $\alpha, \beta \in \D\setminus \{0\}$. Let
\[
\vp_{1} = b_{\alpha} \text{ and } \vp_2 = b_{\beta},
\]
and suppose $\alpha \neq \beta$. Since $\vp_1(\alpha) = b_\alpha(\alpha) = 0$, it follows that
\[
\begin{split}
P_{\vp_1H^2(\D)} \bS(\cdot, \alpha) & = M_{\vp_1} M_{\vp_1}^* \bS(\cdot, \alpha)
\\
& = M_{\vp_1} (\overline{\vp_1(\alpha)} \bS(\cdot, \alpha))
\\
&= 0,
\end{split}
\]
and hence
\begin{align*}
[P_{\vp_{1}H^2(\D)},P_{\vp_{2}H^2(\D)}]\bS(\cdot, \alpha) & = P_{\vp_{1}H^2(\D)}P_{\vp_{2}H^2(\D)} \bS(\cdot, \alpha)
\\
& = P_{\vp_{1}H^2(\D)} M_{\vp_2} M_{\vp_2}^* \bS(\cdot, \alpha)
\\
& = \overline{\vp_2(\alpha)} P_{\vp_1 H^2(\D)}(\vp_2 \bS(\cdot, \alpha))
\\
& = \overline{b_\beta(\alpha)} P_{\vp_1 H^2(\D)}(\vp_2 \bS(\cdot, \alpha)).
\end{align*}
We claim that $P_{\vp_1 H^2(\D)}(\vp_2 \bS(\cdot, \alpha)) \neq 0$. Indeed, if $P_{\vp_1 H^2(\D)}(\vp_2 \bS(\cdot, \alpha)) = 0$, then
\[
0 = P_{\vp_1 H^2(\D)}(\vp_2 \bS(\cdot, \alpha)) = M_{\vp_1} M_{\vp_1}^* (\vp_2 \bS(\cdot, \alpha)).
\]
Since $M_{\vp_1}$ is an isometry, it follows that $M_{\vp_1}^* (\vp_2 \bS(\cdot, \alpha)) = 0$, so that
\[
b_{\beta} \bS(\cdot, \alpha) = \vp_2 \bS(\cdot, \alpha) \in \ker M_{b_\alpha}^*.
\]
On the other hand, a simple (and well-known computation) shows that
\[
\ker M_{b_\alpha}^* = \mathbb{C} \bS(\cdot, \alpha).
\]
We thus have
\[
b_{\beta} \bS(\cdot, \alpha) = c \bS(\cdot, \alpha),
\]
for some nonzero scalar $c$, and thus $b_\beta = c$, which is a contradiction. Finally, $\alpha \neq \beta$ implies that $b_\beta(\alpha) \neq 0$, and hence
\[
[P_{\vp_{1}H^2(\D)},P_{\vp_{2}H^2(\D)}] \neq 0.
\]
In the present scenario, it can be observed that neither function $\vp_1$ is a divisor of function $\vp_2$, nor is function $\vp_2$ a divisor of function $\vp_1$. Hence, the aforementioned conclusion is further substantiated by Theorem \ref{thm 1 var proj}.
\end{ex}

The following example illustrates Theorem \ref{thm n var proj}.

\begin{ex}
Fix $\alpha, \beta \in \D$. Define inner functions $\vp_1, \vp_2$ in $H^\infty(\D^2)$ by
\[
\vp_{1}(z_1, z_2) = z_{1} b_{\alpha}(z_1),
\]
and
\[
\vp_2(z_1, z_2) = z_{1}b_{\beta}(z_2),
\]
for all $(z_1,z_2) \in \D^2$. We have
\[
P_{\vp_{1}H^2(\D^2)} =  M_{z_1} M_{b_\alpha} M_{b_\alpha}^* M_{z_1}^*,
\]
and
\[
P_{\vp_{2}H^2(\D^2)} =  M_{z_1} M_{b_\beta} M_{b_\beta}^* M_{z_1}^*.
\]
Therefore
\[
[P_{\vp_{1}H^2(\D^2)},P_{\vp_{2}H^2(\D^2)}] = M_{z_1}\big(M_{b_{\alpha}}M_{b_{\alpha}}^* M_{b_{\beta}}M_{b_{\beta}}^* - M_{b_{\beta}}M_{b_{\beta}}^* M_{b_{\alpha}}M_{b_{\alpha}}^*\big)M_{z_1}^*.
\]
Since $b_\alpha$ and $b_\beta$ are separated functions, by \eqref{eqn: dc of TO}, it follows that $M_{b_{\alpha}}^* M_{b_{\beta}} = M_{b_{\beta}} M_{b_{\alpha}}^*$. Then
\[
M_{b_{\alpha}}M_{b_{\alpha}}^* M_{b_{\beta}}M_{b_{\beta}}^* - M_{b_{\beta}}M_{b_{\beta}}^* M_{b_{\alpha}}M_{b_{\alpha}}^* = 0,
\]
hence, by summing
\[
[P_{\vp_{1}H^2(\D^2)},P_{\vp_{2}H^2(\D^2)}] = 0.
\]
In this case, neither $\vp_1$ nor $\vp_2$ divides the other. In contrast to the $n=1$ case, however, the projections commute here because, according to Theorem \ref{thm n var proj},
\[
\psi(z_1, z_2) = z_1 \qquad ((z_1, z_2) \in \D^2),
\]
is a common factor and $b_{\alpha}$ and $b_{\beta}$ are separated functions.
\end{ex}

\smallskip

\noindent\textsf{Acknowledgement:} We are truly appreciative of the referee's attentive reading and insightful comments. The first named author acknowledges IIT Bombay for its warm hospitality. The research of the first named author is supported by the institute post-doctoral fellowship of IIT Bombay. The research of the second named author is partially supported by SG/IITH/F316/2022-23/SG-151, and DST/INSPIRE/04/2021/002101.
The research of the third named author is supported in part by TARE (TAR/2022/000063), SERB (DST), Government of India.

\end{document}